\documentclass[12pt]{amsart}
\usepackage{amssymb}
\usepackage{amsmath}
\usepackage{latexsym}
\usepackage{amsthm}
\usepackage{graphicx}
\usepackage{relsize}
\makeatletter
\def\thmhead@plain#1#2#3{%
 \thmname{#1}\thmnumber{\@ifnotempty{#1}{ }#2}%
 \thmnote{ \the\thm@notefont(#3)}}
\let\thmhead\thmhead@plain
\def\swappedhead#1#2#3{%
 \thmnumber{#2}\thmname{\@ifnotempty{#2}{. }#1}%
 \thmnote{ \the\thm@notefont(#3)}}
\makeatother

\theoremstyle{definition} 
\newtheorem{definition}{Definition}[section]
\newtheorem{remark}[definition]{Remark}

\theoremstyle{plain}      
\newtheorem{proposition}[definition]{Proposition}
\newtheorem{theorem}[definition]{Theorem}
\newtheorem{corollary}[definition]{Corollary}
\newtheorem{lemma}[definition]{Lemma}

\DeclareMathOperator{\Cov}{Cov}
\begin{document}
\title[Overlapping Window Tests]{Overlapping Window Tests for Correlation and Trend}
\author[A. Alhakim]{Abbas Alhakim\\Department of Mathematics\\
American University of Beirut\\
Beirut, Lebanon\\
}

\email{aa145@aub.edu.lb}


\date{May, 2026.}

\begin{abstract}
We develop a general framework for constructing and analyzing overlapping sliding-window statistics for dependence and trend detection. For a fixed window size, the overlapping blocks form a Markov chain, and the asymptotic variance of any centered window statistic is determined by the covariance operator of this chain. Using its spectral structure, we obtain an orthogonal decomposition of the space of window functions into components associated with different overlap levels. This leads to a natural notion of incremental dependence information: the part of a statistic that captures exactly the new information introduced by enlarging the window.

We give an explicit procedure for extracting these components and apply the method to two classes of examples. For correlation detection, we study symmetric polynomial window functions and identify their informative projected part. For trend detection, we analyze localized rank-based statistics and isolate the contribution of the largest newly introduced lag. The examples also show that different statistics may exhibit different local detection scales, including nonclassical ones. The same viewpoint leads to a natural quantitative measure of incremental information, which can be used to assess how much new dependence structure is captured as the window size increases, and to guide scale selection. Overall, the paper provides a systematic method for designing overlapping-window tests and deriving their asymptotic normalization and local behavior.
\end{abstract}

\maketitle 
\section{Introduction}

Many procedures for detecting dependence, trend, and local structure in sequential
data are naturally formulated through functions of consecutive observations. Given
observations $X_1,X_2,\ldots,X_n$, one considers statistics of the form
\begin{equation*}
\sum_{i=1}^{n-k+1} f(X_i,\ldots,X_{i+k-1}),
\end{equation*}
where $f$ is a measurable function of $k$ consecutive observations. Such
sliding-window statistics are useful because they can detect local structure that
may not be visible through global summaries or marginal information alone.

Sliding-window methods arise in several classical and modern settings, including
overlapping serial tests of randomness \cite{Good1953,LSW2002,gM85},
approximate entropy (ApEn) \cite{Pincus1991,PS96}, and rank-based procedures for
trend detection \cite{Cabilio2013,Kendall1955,Mann1945}. Although these procedures
differ in construction and purpose, they share a common feature: increasing the
window size is intended to capture additional dependence or ordering information.
This raises a basic question. When the window size is enlarged from $k-1$ to $k$,
what part of the statistic represents genuinely new information, and what part
merely repeats structure already detectable using smaller windows?

This paper provides a framework for answering this question for overlapping
sliding-window statistics under an independent null hypothesis. Even when the
underlying observations are independent, the overlapping windows are dependent
because consecutive windows share observations. This overlap induces a covariance
operator, whose spectral structure determines the asymptotic variance of the
statistic. We use this operator to decompose window functions into orthogonal
components and identify the part that is orthogonal to all statistics based on
windows of size $k-1$ or smaller. We interpret this part as the incremental
dependence information introduced by enlarging the window.

A central structural outcome is that the incremental information is exactly the
eigenvalue-one component of the asymptotic covariance operator. This viewpoint
appears implicitly in earlier work on overlapping serial statistics; in particular,
the $L^2$ framework of \cite{AKM} already points in this direction, while the
recent paper \cite{AA2024} develops an adjacent idea for component decomposition of
overlapping Pearson-type statistics. The novelty of the present work is the
systematic isolation of the eigenvalue-one component as the incremental part,
together with a constructive decomposition scheme for statistical construction and
local asymptotic analysis.

The framework allows one to begin with a candidate window function and extract its
informative incremental component by orthogonal projection. It also provides
explicit variance formulas for normalization and leads to a natural scale-selection
criterion: by quantifying the relative size of the incremental component, one may
assess how much genuinely new information is gained as the window size increases.

We illustrate the method in two directions. For correlation detection, we study
symmetric polynomial window functions and show that their incremental component has
a tractable endpoint-interaction structure. For trend detection, we consider
localized rank-based procedures and show that the decomposition isolates the new
contribution introduced by the largest lag within the window. These examples also
show that different window statistics may operate on different local detection
scales, including nonclassical ones, formulated through local alternatives of the
form $\theta_n=\theta_0+c n^{-\gamma}$.

The paper is organized as follows. Section 2 introduces the overlapping-window
setting and the asymptotic variance problem. Section 3 develops the decomposition
and identifies the incremental component. Section 4 applies the method to
correlation and trend statistics and introduces the relative incremental scale
index. Section 5 reports simulation studies for size and local
power, including comparisons between full statistics and their incremental
components. Section 6 concludes with potential future directions.


\section{The Overlapping Framework and Asymptotic Variance Structure}

\subsection{Overlapping window statistics}

Let $\{X_i\}_{i=1}^n$ be a sequence of independent and identically distributed random variables defined on a probability space $(\mathcal{X},\mathcal{F},\nu)$. 
For a fixed integer $k \ge 1$, define the overlapping window process
\[
Y_i = (X_i, X_{i+1}, \dots, X_{i+k-1}), 
\qquad i = 1, \dots, n-k+1.
\]
 Let $\mathcal{H}_k = L^2(\mathcal{X}^k,\pi_k)$ denote the Hilbert space of square-integrable window functions with mean zero under the product measure $\pi_k= \nu^{\otimes k}$, that is, $\mathbb{E}_{\pi_k}[f] = 0$.
For $f\in\mathcal{H}_k$, 
we consider the normalized sliding-window statistic
\begin{equation}\label{E:f_statistic}
S_f(n) = \frac{1}{\sqrt{n-k+1}} \sum_{i=1}^{n-k+1} f(Y_i).
\end{equation}
\noindent
To simplify the exposition in later sections, we will occasionally abuse notation and refer to a function $f$ as a statistic, when we mean the corresponding sliding-window statistic $S_f(n)$.

Because consecutive windows share $k-1$ observations, the sequence $\{Y_i\}$ is not independent. 
Instead, it forms a time-homogeneous Markov chain on $\mathcal{X}^k$. 
This overlap induces nontrivial covariance between terms in the partial sum, and therefore the asymptotic variance of $S_f(n)$ is not simply $\mathrm{Var}(f(Y_1))$.

Under mild moment conditions, a central limit theorem holds for the overlapping-window statistic; see~\cite{AKM} for the present setting, and see, for example, \cite{GordinLifsic1978, MeynTweedie2009} for general central limit theorems for additive functionals of Markov chains:

\[
S_f(n) \;\Longrightarrow\; N(0,\sigma^2(f)),
\]
where the asymptotic variance $\sigma^2(f)$ reflects the dependence induced by overlap. 
The principal inferential challenge is therefore the explicit computation and structural understanding of $\sigma^2(f)$.

\subsection{Covariance operator representation}

A general expression for the asymptotic variance is given by the usual
long-run variance formula for additive functionals of a stationary sequence; see,
for example, \cite{GordinLifsic1978,MeynTweedie2009}. In the present notation this gives
\begin{equation}\label{E:Gen_Asymp_Var}
\sigma^2(f)
=
\operatorname{Var}(f(Y_1))
+
2\sum_{r=1}^{k-1}
\operatorname{Cov}\bigl(f(Y_1),f(Y_{1+r})\bigr).
\end{equation}
Note that the sum is finite because the process $\{Y_i\}$ has finite dependence radius, under the independence of $\{X_i\}$.
Let $P$ denote the transition operator of the Markov chain $\{Y_i\}$, given by
\[
Pf(x_1,\dots,x_k) 
=
\int f(x_2,\dots,x_k,\xi)\,\nu(d\xi).
\]
The chain reaches stationarity in $k$ steps, and its stationary distribution is $\pi_k$. 
The asymptotic variance admits the operator representation, $\sigma^2(f) = \langle B_k f , f \rangle$, as an inner product in $\mathcal H_k$,
where
\[
B_k = I + P + P^* + \cdots + P^{k-1} + (P^*)^{k-1}.
\]

This covariance representation was established in \cite{AKM} in the context of overlapping Markov chains and related nilpotent structures. 
Thus inference for sliding-window statistics reduces to understanding the structure of the covariance operator $B_k$.

\begin{definition}\label{D:S_l}
For an integer $k\ge 1$ and $\ell=1,\dots,k$, define the auxiliary subspace $\mathcal{S}_\ell$ of $\mathcal{H}_k$ 
to consist of all functions of the form
\[
f(x_1,\dots,x_k)=\phi_1(x_1,\dots,x_{k-\ell+1})+\phi_2(x_2,\dots,x_{k-\ell+2})+\cdots+\phi_\ell(x_\ell,\dots,x_k),
\]
where, for each $j=1,\dots,\ell$, the function $\phi_j\in L^2(X^{k-\ell+1},\pi_{k-\ell+1})$
satisfies
\[
\int_\mathcal{X} \phi_j(x_1,\dots,x_{k-\ell+1})\,\nu(dx_1)=0,
\qquad
\int_\mathcal{X} \phi_j(x_1,\dots,x_{k-\ell+1})\,\nu(dx_{k-\ell+1})=0.
\]
\end{definition}

\begin{remark}
Under the product measure, the spaces $S_l$ are mutually orthogonal. 
Indeed, if $\ell\neq \ell'$, then in any product of a summand from $S_\ell$ with one from $S_{\ell'}$, there is an endpoint variable of the longer block that appears only in one factor; the side-centering condition makes that factor mean-zero in this variable, so integrating it out forces the inner product to vanish.
\end{remark}

\noindent
The following is Theorem 2 of \cite{AKM}.
\begin{theorem}\label{T:AKM}
For each $\ell=1,\dots,k$, the space $\mathcal{S}_\ell$ is invariant under $B_k$ and contains the
eigenspace $\mathcal{L}_{\ell}$ corresponding to the eigenvalue $\ell$. Moreover, $\mathcal{S}_1=\mathcal{L}_1$, and for
$\ell>1$,
\[
\mathcal{S}_\ell = \mathcal{L}_\ell \oplus (\mathcal{S}_\ell\cap \ker B_k).
\]
In particular, every function in $\mathcal{S}_\ell$ decomposes uniquely into an $\mathcal{L}_\ell$-component
and a kernel component.

Furthermore, for a function $f$  in $\mathcal{S}_\ell,\;\ell>1$, the component of $f$ in $\mathcal{L}_\ell$ is given by 
\[
\displaystyle\mathlarger{f}_\ell(x_1,\ldots,x_k)=\phi(x_1,\ldots,x_{k-\ell+1})+\phi(x_2,\ldots,x_{k-\ell+2})+\ldots+\phi(x_\ell,\ldots,x_k),
\]
\noindent where 
$\phi(x_1,\ldots,x_{k-\ell+1})=\displaystyle\frac{\phi_1(x_1,\ldots,x_{k-\ell+1})+\ldots+\phi_\ell(x_1,\ldots,x_{k-\ell+1})}\ell$, and the marginal functions $\phi_j$ are as in Definition~\ref{D:S_l}.
\end{theorem}

\subsection{Spectral structure and levels of dependence}

A key structural fact, proved in \cite{AA04,AKM}, is that the operator $B_k$ has a remarkably simple spectrum: its only eigenvalues are the integers
\[
0,1,2,\dots,k.
\]
Since $B_k$ is self-adjoint, the Hilbert space $\mathcal{H}_k$ admits the orthogonal decomposition
\[
\mathcal{H}_k=\mathcal{L}_0\oplus\mathcal{L}_1\oplus\cdots\oplus \mathcal{L}_k,
\]

where each eigenspace $\mathcal{L}_\ell$ corresponds to eigenvalue $\ell$ and is associated to a distinct level of dependence induced by window overlap.
Thus every $f\in \mathcal{H}_k$ admits the unique decomposition
\[
f=f_0+f_1+\cdots+f_k,
\]
\noindent
where $f_\ell\in\mathcal{L}_\ell$.
Because $B_k$ acts diagonally on this decomposition, the asymptotic variance decomposes additively:
\begin{equation}\label{E:variance}
\sigma^2(f)
=
\sum_{\ell=1}^k \ell \, \|f_\ell\|^2,
\end{equation}
\noindent
where $f_\ell$ denotes the orthogonal projection of $f$ onto $\mathcal{ L}_\ell$.


\section{Systematic Decomposition and Extraction of Incremental Dependence}

To obtain the components $f_\ell$ explicitly, it is convenient to make use of the orthogonal auxiliary
subspaces $\mathcal{S}_\ell$ given in Definition~\ref{D:S_l}. Each $\mathcal{S}_\ell$ contains $\mathcal{L}_\ell$ and differs from it only by a possible
kernel contribution:
\[
\mathcal{S}_\ell=\mathcal{L}_\ell\oplus\mathcal{L}_0^{(\ell)},\qquad \mathcal{L}_0^{(\ell)}:=\mathcal{S}_\ell\cap \ker(B_k).
\]
Thus the spaces $\mathcal{S}_\ell$ provide a constructive realization of the spectral decomposition.

\subsection{A Constructive auxiliary decomposition}\label{Ss:auxiliary_Decomp}

To achieve this decomposition, we integrate $f$ over successive boundary variables.
For $2 \le \ell \le k$, define the marginal integrals
\[
f_\ell^{(1)}(x_1,\dots,x_{k-\ell+1})
=
\int f(x_1,\dots,x_k)\,\nu(dx_{k-\ell+2})\cdots\nu(dx_k),
\]
\[
f_\ell^{(2)}(x_2,\dots,x_{k-\ell+2})
=
\int f(x_1,\dots,x_k)\,\nu(dx_1)\nu(dx_{k-\ell+3})\cdots\nu(dx_k),
\]
and so forth, integrating over all but $k-\ell+1$ consecutive variables.

For $\ell=1$, define $f_1^{(1)}= f$.

The next theorem gives a successive construction of the auxiliary decomposition.

\begin{theorem}\label{T:projection}
Let $f\in \mathcal{H}_k$. For each $\ell=1,\dots,k$, define
$\displaystyle
s_\ell:=\sum_{i=1}^\ell g_\ell^{(i)},
$
where
\[
g_k^{(i)}=f_k^{(i)},
\]
\[
g_{k-1}^{(i)}(x_i,x_{i+1})
=
f_{k-1}^{(i)}(x_i,x_{i+1})-f_k^{(i)}(x_i)-f_k^{(i+1)}(x_{i+1}),
\]
and, for $1\le \ell\le k-2$,
\[
g_\ell^{(i)}(x_i,\dots,x_{i+k-\ell})
=
f_\ell^{(i)}(x_i,\dots,x_{i+k-\ell})
+f_{\ell+2}^{(i+1)}(x_{i+1},\dots,x_{i+k-\ell-1})
\]
\[
\hspace{4cm}
-f_{\ell+1}^{(i)}(x_i,\dots,x_{i+k-\ell-1})
-f_{\ell+1}^{(i+1)}(x_{i+1},\dots,x_{i+k-\ell}).
\]
Then $s_\ell\in \mathcal S_\ell$ for each $\ell$, and we get the auxiliary decomposition
\[
f=s_1+s_2+\cdots+s_k.
\]
Consequently, writing
\[
s_\ell=f_\ell+f_0^{(\ell)},
\qquad
f_\ell\in \mathcal L_\ell,\quad f_0^{(\ell)}\in \mathcal L_0^{(\ell)},
\]
one obtains the orthogonal spectral decomposition
\[
f=f_0+f_1+\cdots+f_k,
\qquad
f_0=\sum_{\ell=1}^k f_0^{(\ell)}.
\]
In particular, since $\mathcal S_1=\mathcal L_1$, one has $s_1=f_1$.
\end{theorem}

\medskip

\begin{proof}
For $2\le \ell\le k$, integration in the definition of $f_\ell^{(i)}$ is over $\mathcal{X}^{\,\ell-1}$.
By Fubini's theorem, for $\ell\le k-1$ we have
\begin{equation}\label{eq:left-marginal}
\int_\mathcal{X} f_\ell^{(i)}(x_i,\dots,x_{i+k-\ell})\,\nu(dx_i)
=
f_{\ell+1}^{(i+1)}(x_{i+1},\dots,x_{i+k-\ell}),
\end{equation}
and
\begin{equation}\label{eq:right-marginal}
\int_\mathcal{X} f_\ell^{(i)}(x_i,\dots,x_{i+k-\ell})\,\nu(dx_{i+k-\ell})
=
f_{\ell+1}^{(i)}(x_i,\dots,x_{i+k-\ell-1}).
\end{equation}

Since $E_{\pi_k}[f]=0$, it follows immediately that $\displaystyle\int_\mathcal{X} f_k^{(i)}(x_i)\,\nu(dx_i)=0$,

so that $\mathlarger{s}_k=\displaystyle\sum_{i=1}^k f_k^{(i)}$ belongs to $\mathcal S_k$.

Next, applying \eqref{eq:left-marginal} and \eqref{eq:right-marginal} to the definition of
$g_\ell^{(i)}$, we obtain for $\ell\le k-1$
\[
\int_\mathcal{X} g_\ell^{(i)}(x_i,\dots,x_{i+k-\ell})\,\nu(dx_i)=0,
\qquad
\int_\mathcal{X} g_\ell^{(i)}(x_i,\dots,x_{i+k-\ell})\,\nu(dx_{i+k-\ell})=0.
\]
Hence each $g_\ell^{(i)}$ has vanishing left and right marginals on its supporting block,
and therefore
\[
s_\ell=\sum_{i=1}^\ell g_\ell^{(i)}\in \mathcal S_\ell,
\qquad \ell=1,\dots,k.
\]

It remains to prove that the $s_\ell$ sum to $f$. For $1\le \ell\le k-2$, let
\begin{equation}\label{eq:telescoping-claim}
\sum_{m=1}^\ell s_m
=
f-\sum_{j=1}^{\ell+1} f_{\ell+1}^{(j)}+\sum_{j=2}^{\ell+1} f_{\ell+2}^{(j)}.
\end{equation}
For $\ell=1$, this is exactly the formula for $s_1$. Assume \eqref{eq:telescoping-claim}
holds for some $\ell<k-2$. Using the equivalent representation
\[
s_{\ell+1}
=
\sum_{j=1}^{\ell+1} f_{\ell+1}^{(j)}
-\sum_{j=1}^{\ell+2} f_{\ell+2}^{(j)}
-\sum_{j=2}^{\ell+1} f_{\ell+2}^{(j)}
+\sum_{j=2}^{\ell+2} f_{\ell+3}^{(j)},
\]
we get
\[
\sum_{m=1}^{\ell+1} s_m
=
\left(
f-\sum_{j=1}^{\ell+1} f_{\ell+1}^{(j)}+\sum_{j=2}^{\ell+1} f_{\ell+2}^{(j)}
\right)
+
\left(
\sum_{j=1}^{\ell+1} f_{\ell+1}^{(j)}
-\sum_{j=1}^{\ell+2} f_{\ell+2}^{(j)}
-\sum_{j=2}^{\ell+1} f_{\ell+2}^{(j)}
+\sum_{j=2}^{\ell+2} f_{\ell+3}^{(j)}
\right),
\]
which simplifies to
\[
\sum_{m=1}^{\ell+1} s_m
=
f-\sum_{j=1}^{\ell+2} f_{\ell+2}^{(j)}+\sum_{j=2}^{\ell+2} f_{\ell+3}^{(j)}.
\]
This proves \eqref{eq:telescoping-claim} by induction.

Finally, substituting $\ell=k-2$ into \eqref{eq:telescoping-claim} and then adding the
explicit expressions for $s_{k-1}$ and $s_k$, we obtain $\displaystyle\sum_{l=1}^k s_\ell=f$.

Thus $f$ is decomposed as a sum of elements $s_\ell\in\mathcal S_\ell$. Theorem~\ref{T:AKM} states that each
$\mathcal S_\ell$ splits as
\[
\mathcal S_\ell=\mathcal L_\ell\oplus \mathcal L_0^{(\ell)},
\]
For some $\mathcal L_0^{(\ell)}\subset Ker(B_k)$.  This theorem also shows that, for each $\ell>1$, 
\[
s_\ell=f_\ell+f_0^{(\ell)},
\qquad
f_\ell\in \mathcal L_\ell,\quad f_0^{(\ell)}\in \mathcal L_0^{(\ell)}.
\]
Summing over $\ell$ yields
\[
f=f_0+f_1+\cdots+f_k,
\qquad
f_0=\sum_{\ell=1}^k f_0^{(\ell)}.
\]
Since $\mathcal S_1=\mathcal L_1$, the first constructed component satisfies 
$s_1=f_1$.
\end{proof}

\begin{definition}
A function $f(x_1,\dots,x_k)$ is called symmetric (permutation invariant) if
\[
f(x_{\sigma(1)},\dots,x_{\sigma(k)})=f(x_1,\dots,x_k)
\]
for every permutation $\sigma$ of $\{1,\dots,k\}$.
\end{definition}

The next result shows that for the class of symmetric functions, the auxiliary decomposition already coincides with the spectral one.

\begin{theorem}\label{prop:perm-invariant}
Let $f(x_1,\dots,x_k)$ be a symmetric function. Then, the
auxiliary decomposition of $f$ coincides with the spectral decomposition. That is, for each $\ell=1,\dots,k$, $s_\ell$ is an eigenfunction of eigenvalue $\ell$. In particular, $f$ is orthogonal to the kernel of $B_k$.
\end{theorem}

\begin{proof}
Fix $\ell$ and $i$ with $1\le i\le \ell\le k$. By definition,
\[
f_\ell^{(i)}(x_i,\dots,x_{i+k-\ell})
=
\int_{\mathcal{X}^{\,\ell-1}}
f(x_1,\dots,x_k)\,
\nu(dx_1)\cdots \nu(dx_{i-1})\,\nu(dx_{i+k-\ell+1})\cdots \nu(dx_k).
\]
Define a permutation $\sigma$ of $\{1,\dots,k\}$ by
\[
\sigma(j)=
\begin{cases}
i+j-1, & 1\le j\le k-\ell+1,\\[4pt]
j-k+\ell-1, & k-\ell+2\le j\le k-\ell+i,\\[4pt]
j, & k-\ell+i+1\le j\le k.
\end{cases}
\]
Since $f$ is permutation invariant,
\[
f(x_1,\dots,x_k)=f(x_{\sigma(1)},\dots,x_{\sigma(k)}).
\]
Substituting this into the above integral and observing that the integrated variables are
exactly those outside the block $(x_i,\dots,x_{i+k-\ell})$, we obtain
\[
f_\ell^{(i)}(x_i,\dots,x_{i+k-\ell})=f_\ell^{(1)}(x_i,\dots,x_{i+k-\ell}).
\]
Thus $f_\ell^{(i)}$ is independent of $i$. The first claim follows by  the last assertion of  Theorem~\ref{T:AKM}. The second claim follows by Theorem~\ref{T:projection}.

\end{proof}

The next corollary follows immediately from the preceding theorem and the spectral variance formula. It explains why symmetric functions are especially convenient in practice, since in this case the auxiliary and spectral components coincide.

\begin{corollary}\label{C:perm_invariant}
Let $f\in \mathcal H_k$ be symmetric, and let $\displaystyle
f=\sum_{\ell=1}^k s_\ell$
be its auxiliary decomposition.
Then its asymptotic variance is given by $\displaystyle\sigma^2(f)=\sum_{\ell=1}^k \ell\,\|s_\ell\|^2$.
\end{corollary}

\subsection{Incremental Component}

As mentioned in the introduction, we seek to isolate the
information introduced by enlarging the window size. 

\noindent
\begin{definition} (Incremental Dependence Information)
For $f \in \mathcal{H}_k$, the incremental dependence information at window size $k$ is defined as the orthogonal projection
\[
f^{\mathrm{inc}} := \mathrm{Proj}_{\mathcal{ L}_1} f.
\]

Equivalently, $f^{\mathrm{inc}}$ is the unique element in $\mathcal{ L}_1$ such that
\[
\langle f - f^{\mathrm{inc}}, h \rangle = 0
\qquad
\text{for all } h \in \mathcal{ L}_1.
\]
\end{definition}

By the spectral variance formula (\ref{E:variance}), the variance contribution of this incremental component is
$
\sigma^2(f^{\mathrm{inc}})=\|f^{\mathrm{inc}}\|^2$, 
which quantifies the amount of dependence detectable uniquely at span $k$.


\subsection*{An intrinsic characterization of the incremental subspace}

Define the conditional expectation operators
\[
(\mathsf E_L f)(x_1,\dots,x_{k-1})=\int f(x_1,\dots,x_k)\,\nu(dx_k),
\\
(\mathsf E_R f)(x_2,\dots,x_k)=\int f(x_1,\dots,x_k)\,\nu(dx_1),
\]
and let
\[
\mathcal{W}_{k-1}
=
\left\{
u(x_1,\dots,x_{k-1})+v(x_2,\dots,x_k)
:\ u,v\in L^2(\mathcal{X}^{k-1},\pi_{k-1})
\right\}.
\]

\begin{proposition}[Incremental subspace as an orthogonal complement]
For $f\in\mathcal{H}_k$, the following are equivalent:
\begin{enumerate}
\item $f \in \mathcal{W}_{k-1}^{\perp}$;
\item $\mathsf E_L f = 0$ and $\mathsf E_R f = 0$ (in $L^2$);
\item For $\pi_k$-a.e.\ $(x_1,\dots,x_{k-1})$ and $(x_2,\dots,x_k)$,
\[
\int f(x_1,\dots,x_{k-1},x_k)\,\nu(dx_k)=0,
\qquad
\int f(x_1,x_2,\dots,x_k)\,\nu(dx_1)=0.
\]
\end{enumerate}
\end{proposition}

\noindent
\emph{Sketch.}
If $f\perp \mathcal{W}_{k-1}$ then $f$ is orthogonal to every function of $(X_1,\dots,X_{k-1})$ and to every function of $(X_2,\dots,X_k)$, which is equivalent to $\mathsf E_L f=0$ and $\mathsf E_R f=0$.
Conversely, if both conditional expectations vanish then $f$ is orthogonal to each summand in $\mathcal{W}_{k-1}$ and hence to $\mathcal{W}_{k-1}$.
\qed

The preceding proposition gives an intrinsic characterization of the incremental subspace.
Since $\mathcal S_1=\mathcal L_1$, the new information introduced at window size $k$ is precisely
the orthogonal complement $\mathcal W_{k-1}^{\perp}$ of the smaller-window space $\mathcal W_{k-1}$.
Equivalently, $\mathcal L_1$ consists exactly of those functions whose left and right conditional
expectations vanish. Thus the hierarchy across window sizes is already implicit in the identity
\[
\mathcal L_1^{(k)}=\mathcal W_{k-1}^{\perp},
\]
which shows that enlarging the window adds exactly the component orthogonal to all statistics based on smaller windows.

\subsection*{Relation to entropy-based incremental measures}

The notion of incremental dependence considered here is conceptually related to entropy-based measures such as approximate entropy (ApEn) \cite{Pincus1991, PS96}, which compare pattern frequencies at successive embedding dimensions. 
Entropy-based methods quantify the increase in unpredictability when enlarging the window. 
In contrast, the present framework operates in the quadratic geometry of $L^2$ and isolates the orthogonal complement of functions representable at smaller window sizes. 
Both approaches reflect the same structural principle: increasing the window expands the model space, and the relevant quantity is the part not explained by smaller windows.

\subsection*{Comparison with Hoeffding-type decompositions}

The decomposition presented here bears a formal resemblance to the Hoeffding or ANOVA decomposition for symmetric functions of independent variables. 
In both cases, a function is expressed as an orthogonal sum of components associated with increasing interaction order.

However, the present setting differs in two essential respects. 
First, the variables entering $f$ are not independent; they arise from overlapping blocks of a single sequence, and the dependence structure is encoded through the covariance operator $B_k$. 
Second, the decomposition is determined by the spectral structure of this operator rather than by combinatorial symmetry alone. 
The resulting subspaces $\mathcal{ L}_\ell$ reflect levels of dependence induced by window overlap, rather than interaction order among independent arguments.

Thus while the hierarchical structure is reminiscent of classical ANOVA decompositions, the mechanism and interpretation are fundamentally tied to the overlapping Markov structure. 
For background on the classical Hoeffding decomposition and its ANOVA-type interpretation for symmetric functions of independent variables, see \cite{Serfling1980, vdVaart1998}.

\section{Applications: Designing Sliding-Window Tests}

We illustrate the framework on two classes of sliding-window statistics, respectively targeting correlation and monotone trend. In each
case, we begin with a natural window function and then extract its incremental
component using the theory developed above.

\subsection{Detecting Local Correlation}\label{SS:cov}

We first consider alternatives exhibiting short-range correlation, 
such as autoregressive-type dependence. A natural class of window 
functions in this setting consists of symmetric polynomial functions 
that capture interactions between observations at different positions 
within the window.

For fixed window size $k$, let $g(x_1,\dots,x_k)$ be the symmetric polynomial function defined in~(\ref{E:g l k}). Applying the projection technique of Section~3, we obtain its 
incremental component $g_1$.

To motivate the choice of this function, we consider $\tilde{f}(x_1,x_2)=x_1x_2$ which may be suggested to test for correlation between consecutive terms in the sequence $X_1,X_2,\cdots X_n$. The first step in the decomposition is to centralize $\tilde{f}$ to $f(x_1,x_2)=x_1x_2-\mu^2$ which is in $\mathcal{H}_2$. Since $k=2$, Theorem~\ref{T:projection} immediately gives the $\mathcal{L}_1$ component as
\begin{eqnarray*}
f_1(x_1,x_2) &=& f(x_1,x_2)-f_2^{(1)}(x_1)-f_2^{(2)}(x_2)\\
		&=& (x_1x_2-\mu^2)-(\mu x_1-\mu^2)-(\mu x_2-\mu^2)\\
		&=& (x_1-\mu)(x_2-\mu)\label{E:COV}
\end{eqnarray*}
Therefore, the statistic $\displaystyle\frac1{\sqrt{n-k+1}}\sum_{i=1}^{n-1}f_1(x_i,x_{i+1})$ estimates $\Cov(X_1,X_2)$, the covariance between consecutive terms. This suggests that the new information, resulting from incrementing the window size to $k$ and conveyed by the $\mathcal{L}_1$,  contained in the mixed terms $X_iX_{i+1}$ is sensitive to the presence of correlation between consecutive terms.

A generalization to higher window sizes of $k$ uses the $\mathcal{L}_1$ projection of the following elementary symmetric polynomial

\begin{equation}\label{E:g l k}
 g(x_1,\ldots, x_k)=\sum\prod_{j=1}^lx_{i_j}
\end{equation}
where the sum is taken over all monomials of fixed order $l\leq k$. That is, over all sets $\{i_1,\ldots,i_l\}$ of $\{1,\ldots,k\}$.

\begin{proposition}\label{P:g}
Let $2\le l\le k$ and let $g$ be the elementary symmetric polynomial of degree $l$ in
$x_1,\ldots,x_k$, given in (\ref{E:g l k}). Then the $\mathcal L_1$ component of $g$ has the form
\[
g_1(x_1,\ldots,x_k)
=
(x_1-\mu)(x_k-\mu)\,\widetilde g(x_2,\ldots,x_{k-1}),
\]
where $\widetilde g$ is the elementary symmetric polynomial of degree $l-2$ in the
intermediate variables $x_2,\ldots,x_{k-1}$.
\end{proposition}

\begin{proof}
Since subtracting the constant ${k\choose l}\mu^l$ does not affect the $\mathcal L_1$
component, we work directly with $g$. Write
\[
g=A_{k,l}(x_1+x_k)+B_{k,l}x_1x_k+C_{k,l},
\]
where $A_{k,l}$, $B_{k,l}$, and $C_{k,l}$ are respectively the elementary symmetric
polynomials of degrees $l-1$, $l-2$, and $l$ in the intermediate variables
$x_2,\ldots,x_{k-1}$.

By Theorem~\ref{T:projection}, $g_1(x_1,\ldots,x_k)=g-g_2^{(1)}-g_2^{(2)}+g_3^{(2)}$, 
where direct computation shows that 
\[
g_2^{(1)}=A_{k,l}x_1+A_{k,l}\mu+B_{k,l}\mu x_1+C_{k,l},
\]
\[
g_2^{(2)}=A_{k,l}\mu+A_{k,l}x_k+B_{k,l}\mu x_k+C_{k,l},
\]
and
\[
g_3^{(2)}=2A_{k,l}\mu+B_{k,l}\mu^2+C_{k,l}.
\]
Substitution gives
\[
g_1
=
B_{k,l}(x_1x_k-\mu x_1-\mu x_k+\mu^2)
=
B_{k,l}(x_1-\mu)(x_k-\mu).
\]
Since $B_{k,l}$ is precisely the elementary symmetric polynomial of degree $l-2$ in
$x_2,\ldots,x_{k-1}$, the result follows.
\end{proof}

As a partial converse to the above proposition, one can easily establish that any function of the form $(x_1-\mu)(x_k-\mu)\tilde{g}$ is a pure $\mathcal{L}_1$ component, when $\tilde{g}$ is a function of the intermediate variables $x_2,\ldots,x_{k-1}$.

This observation provides a direct design principle. We can immediately construct functions that 
 lie entirely in \(\mathcal{L}_1\),
and hence represent pure incremental statistics. Unlike general window
functions, they contain no lower-order or degenerate components, and therefore
isolate the dependence structure introduced uniquely at span \(k\) without
requiring any projection.

Special cases of the above functions are $g(x_1,\ldots,x_k)=x_1x_2\cdots x_k$, which corresponds to $l=k$, with $g_1=x_2\cdots x_{k-1}(x_1-\mu)(x_k-\mu)$. For any $k$, note that the case of $l=2$ reduces to $g_1=(x_1-\mu)(x_k-\mu)$, so a true generalization uses $l\geq3$. The case $l=3$ yields $g_1=(x_1-\mu)(x_k-\mu)(x_2+\ldots+x_{k-1})$, while the case $l=4$ with $k\geq4$ gives 
\begin{equation}\label{E:l=4,k}
g_1(x_1,\ldots,x_k)=\left(x_1-\mu\right)(x_k-\mu)\times\left(\underset{2\leq i< j\leq k-1}{\sum\sum}x_ix_j\right).
\end{equation}

To implement a test statistic $S_f(n)$ given in Equation~(\ref{E:f_statistic}), one needs an explicit expression of the asymptotic variance. The next proposition computes this asymptotic variance for $g_1$ given in (\ref{E:l=4,k}) and will be used in the simulation section below.

\begin{proposition}\label{P:g_Asymp_var}
Let $g_1$ be the $\mathcal{L}_1$ projection of the function $g$ given in Equation~(\ref{E:g l k}) with $k\geq4$ and $l=4$,
where $X_1,\dots,X_k$ are i.i.d.\ with mean $\mu$ and variance $\sigma^2$.
Then the asymptotic variance of $g_1$
is given by
\[
\sigma^2(g_1)
=
\frac{\sigma^4 (k-2)(k-3)}{4}
\left[
(\mu^2+\sigma^2)^2
+
2(k-4)\mu^2(\mu^2+\sigma^2)
+
\frac{(k-4)(k-5)}{2}\mu^4
\right].
\]
\end{proposition}

\begin{proof}
 Let $\displaystyle\alpha=\mathbb E[X_1^2]=\mu^2+\sigma^2$,
and write $S:=\displaystyle\sum_{2\le i<j\le k-1} X_iX_j$,
so that Equation~(\ref{E:l=4,k}) can be written as
\[
g_1=(X_1-\mu)(X_k-\mu)\,S.
\]
By independence,
\[
\|g_1\|^2
=
\mathbb E[(X_1-\mu)^2]\,
\mathbb E[(X_k-\mu)^2]\,
\mathbb E[S^2]
=
\sigma^4\,\mathbb E[S^2].
\]
\noindent
We now compute $\mathbb E[S^2]$ by considering the expansion,
\[
S^2
=
\sum_{2\le i<j\le k-1}\;\sum_{2\le r<s\le k-1} X_iX_jX_rX_s.
\]
The terms are classified according to the overlap between the unordered pairs
$\{i,j\}$ and $\{r,s\}$.

\medskip
\noindent
\emph{(i) Identical pairs.}
There are $\displaystyle\binom{k-2}{2}$ 
such terms, each contributing $
\mathbb E[X_i^2X_j^2]=\alpha^2$.

\medskip
\noindent
\emph{(ii) Pairs sharing exactly one index.}
Choose the common index in $(k-2)$ ways, then choose two distinct remaining
indices from the other $(k-3)$ positions, giving
\[
(k-2)\binom{k-3}{2}
\]
unordered choices. Since either pair may play the role of $(i,j)$ or $(r,s)$
in the double sum, this contributes
\[
2(k-2)\binom{k-3}{2}
\]
terms in total. Each contributes $\displaystyle\mathbb E[X_i^2X_jX_s]=\alpha\mu^2$.

\medskip
\noindent
\emph{(iii) Disjoint pairs.}
Choose four distinct indices from the $(k-2)$ middle variables, then partition
them into two unordered pairs. This gives $\displaystyle3\binom{k-2}{4}$
unordered pairings, and again the two factors in the double sum may be interchanged,
so the total number of such terms is
\[
6\binom{k-2}{4}.
\]
Each contributes $\displaystyle\mathbb E[X_iX_jX_rX_s]=\mu^4$.

\medskip
Summing the three contributions,
\[
\mathbb E[S^2]
=
\binom{k-2}{2}\alpha^2
+
2(k-2)\binom{k-3}{2}\alpha\mu^2
+
6\binom{k-2}{4}\mu^4.
\]
Multiplying by $\sigma^4$ and replacing $\alpha$ gives the stated formula. Finally, since $g_1\in \mathcal{ L}_1$, its asymptotic variance is
$\sigma^2(g_1)=\|g_1\|^2$.
\end{proof}


Proposition~\ref{P:g_Asymp_var} provides the explicit asymptotic variance required 
for normalization of the projected statistic. The variance grows 
polynomially in $k$, reflecting the increasing number of endpoint 
interaction terms as the window expands. 

Thus enlarging the window increases sensitivity to longer-range 
correlation, but only through the incremental component isolated 
by the projection onto $\mathcal{L}_1$.

We record the following variance formula for later use, under the specific $N(0,1)$ assumption.

\begin{lemma}\label{L:g_variance_general}
Let $g$ be as in (\ref{E:g l k}), and suppose
that \(X_1,X_2,\ldots\) are i.i.d. \(N(0,1)\). Then
\[
\sigma^2\bigl(g\bigr)
=
\binom{k}{\ell}+2\binom{k}{\ell+1}.
\]
\end{lemma}

\begin{proof}
Let
$\displaystyle Z_i=g(Y_i)=g(X_i,\ldots,X_{i+k-1})$.
Using Formula (\ref{E:Gen_Asymp_Var}), and the fact that distinct square-free monomials in independent
centered variables are orthogonal, we have
\[
\operatorname{Var}(Z_1)=\binom{k}{\ell}.
\]
For lag \(r\), the two windows overlap in \(k-r\) variables. A pair of monomials
contributes nonzero expectation only when the same \(\ell\) actual variables are
selected from this overlap. Hence
\[
\operatorname{Cov}(Z_1,Z_{1+r})=\binom{k-r}{\ell}.
\]
The result follows by the hockey-stick identity.
\end{proof}

\subsection{Tests for Trend}\label{Sub:Trend}
In order to detect trend in a time series, we assume that the measurements are independent and identically distributed from an absolutely continuous distribution $F(x)$ corresponding to the probability measure $\nu$, which justifies the assumption of no ties in the data. Indeed, a variety of window functions could in principle be used for trend detection. For example, permutation-invariant summaries such as the median, maximum, or minimum may be sensitive to changes in local level or extremal behavior. Such functions, however, do not reflect the temporal ordering of the observations within the window. For this reason, we focus here on the following statistic $K$, whose signed lag comparisons make it directly sensitive to monotone trend,

\begin{equation}\label{E:Ktrend}
K(x_1,\ldots,x_k)=\sum_{\Delta=1}^{k-1}\sum_{j=1}^{k-\Delta}\alpha_{\Delta}\operatorname{sgn}(x_{\Delta+j}-x_j).
\end{equation}

The statistic aggregates signed pairwise comparisons across different lag separations
within the window. The coefficients $\alpha_\Delta$ allow the procedure to emphasize
particular lag distances; for instance, increasing weights assign greater importance to
persistent increases over longer lags.

When $k=n$ and $\alpha_\Delta\equiv 1$, the resulting statistic is essentially the
classical Mann--Kendall statistic. In this sense, $S_K(n)$ may be viewed as a
localized rank-based trend statistic, where comparisons are restricted to observations
that fall within the same window.

In keeping with the general perspective of this paper, we first identify the
incremental component $K_1$, which isolates the information introduced uniquely at
window size $k$. We then return to the full statistic $K$ and compute its asymptotic
variance directly.

The following lemma is used as a repository of the types of integrals needed to obtain the decomposition of $K$. The proof is straightforward and therefore omitted.

\begin{lemma}\label{L:sgn}
For $i<j$ we have

(a) $\int_{\mathcal{X}}\operatorname{sgn}(x_j-x_i)dF(x_i)=2F(x_j)-1$

(b) $\int_{\mathcal{X}}\operatorname{sgn}(x_j-x_i)dF(x_j)=1-2F(x_i)$

(c) $\int_{\mathcal{X}}\operatorname{sgn}(x_j-x_i)F(x_i)dF(x_i)=F^2(x_j)-\frac12$

(d) $\int_{\mathcal{X}}\operatorname{sgn}(x_j-x_i)F(x_j)dF(x_j)=\frac12-F^2(x_i)$

(e) $\int_{\mathcal{X}^2}\operatorname{sgn}(x_j-x_i)dF(x_i)dF(x_j)=0$
\end{lemma}


Using Part (e) of this lemma, we can clearly see that $K$ is centered about zero. Using Parts (a) and (b), we obtain the $\mathcal{L}_1$ component function $K_1=K-K_2^{(1)}-K_2^{(2)}+K_3^{(2)}$ as

\[
K_1(x_1,\ldots,x_k)=\alpha_{k-1}\left(\operatorname{sgn}(x_k-x_1)-2(F(x_k)-F(x_1))\right),
\]

\noindent
where we apply the indexing convention of Subsection~\ref{Ss:auxiliary_Decomp}. This component function shows that the scoring coefficients are not relevant in the context of the $\mathcal{L}_1$ component because only $\alpha_{k-1}$ appears in $K_1$ and it can therefore be equated with 1. As a test statistic, $K_1$ is distribution-free because $F(x_i)$ has a standard uniform $U[0,1]$. However, calculation of $K_1$ requires the explicit form of the {\it parent} distribution and therefore the corresponding test $S_{K_1}$ is no longer nonparametric.   It can still serve, however,  as a test of uniformity.

 A nonparametric test can be obtained using $K$ provided we calculate $\sigma_K^2(n)$. To obtain the asymptotic variance of the unprojected statistic associated with $K$, we use a direct covariance calculation based on its representation as a weighted sum of lagwise sign comparisons.


\begin{proposition}
Consider $\displaystyle S_K(n)=\frac{1}{\sqrt{n-k+1}}
\sum_{i=1}^{n-k+1}K(X_i,\ldots,X_{i+k-1})$, 

where $\{X_t\}$ are i.i.d.\ from a continuous distribution and $K$ is given in Equation~(\ref{E:Ktrend}). Then $S_K(n)$ is asymptotically normal with mean zero and 
asymptotic variance
\[
\sigma_K^2
=
\frac13\sum_{\Delta=1}^{k-1}(k-\Delta)^2\alpha_\Delta^2.
\]
\end{proposition}

\begin{proof}
For $1\le \Delta\le k-1$, let
\[
U_{t,\Delta}=\operatorname{sgn}(X_{t+\Delta}-X_t),
\qquad
T_{\Delta,n}=\sum_{t=1}^{n-\Delta}U_{t,\Delta}.
\]
Then each lag-$\Delta$ comparison appears exactly $k-\Delta$ times in the interior of the
overlapping sum, so
\[
\sum_{i=1}^{n-k+1}K(X_i,\ldots,X_{i+k-1})
=
\sum_{\Delta=1}^{k-1}(k-\Delta)\alpha_\Delta T_{\Delta,n}+O(1).
\]
Hence it suffices to compute the asymptotic covariance structure of the processes
$T_{\Delta,n}$.

Since the distribution is continuous, $\mathbb E[U_{t,\Delta}]=0$ and
$\operatorname{Var}(U_{t,\Delta})=1$. Two variables
$U_{t,\Delta}$ and $U_{s,\Delta'}$ are independent unless the pairs
$(t,t+\Delta)$ and $(s,s+\Delta')$ share an index. When $\Delta=\Delta'$,
the only nonzero off-diagonal covariances arise from $s=t\pm \Delta$, in which case
the shared endpoint appears with opposite sign and
\[
\operatorname{Cov}(U_{t,\Delta},U_{t+\Delta,\Delta})=
\operatorname{Cov}(U_{t,\Delta},U_{t-\Delta,\Delta})=-\frac13.
\]
Therefore
\[
\lim_{n\to\infty}\frac{1}{n}\operatorname{Var}(T_{\Delta,n})
=
1+2\left(-\frac13\right)
=
\frac13.
\]

If $\Delta\neq\Delta'$, the possible shared-endpoint configurations produce two
covariances equal to $+1/3$ and two equal to $-1/3$, so their total contribution
cancels. Hence
\[
\lim_{n\to\infty}\frac{1}{n}\operatorname{Cov}(T_{\Delta,n},T_{\Delta',n})=0,
\qquad \Delta\neq\Delta'.
\]

It follows that
\[
\lim_{n\to\infty}\operatorname{Var}(S_K(n))
=
\frac13\sum_{\Delta=1}^{k-1}(k-\Delta)^2\alpha_\Delta^2,
\]
which is the stated asymptotic variance.
\end{proof}

We note that we did not use the spectral decomposition here, since $K$ is not symmetric. In principle, the spectral components of $K$ could be extracted from its
auxiliary decomposition, but this would be rather lengthy. Fortunately, its asymptotic
variance remains tractable through a direct covariance calculation.

\subsection{A scale-selection index}

The preceding discussion suggests that, as the window size increases, the amount of genuinely new information introduced at the full span may decrease relative to the total variability of the statistic. This motivates a quantitative measure of the incremental contribution at scale \(k\). For a window function \(f\in \mathcal H_k\), we define the relative incremental index
\[
I_k(f)
=
\frac{\|f^{\mathrm{inc}}\|^2}{\sigma^2(f)}.
\]
This quantity measures the proportion of the asymptotic variance that is attributable to the incremental component.
It can be used in two complementary ways. For a nested family of window functions \(f^{(k)}\), the values \(I_k(f^{(k)})\) may be compared across \(k\) to identify scales at which the full-span incremental contribution is relatively large. At a fixed window size, the index also allows comparison between competing statistics: a small value of \(I_k(f)\) indicates that only a small fraction of the asymptotic variance is attributable to the full-span incremental component, while a large value indicates that the statistic is strongly concentrated on that component.

For the two nested families considered  in this article, we show below that \(I_k(f)\) is decreasing in \(k\). This monotonicity gives the index an operational scale-selection interpretation: once \(I_k(f)\) falls below a prescribed threshold, all larger window sizes have no larger relative full-span contribution. Thus the stopping rule is stable, in the sense that the decision to stop is not reversed at higher scales. This is especially useful when the threshold is chosen as a practical relevance threshold for incremental information.


\begin{proposition}
Consider the class of functions $g$, given by (\ref{E:g l k})
where \(l\ge 2\) is fixed, and suppose the underlying variables are i.i.d. \(N(0,1)\).
Then the relative incremental index is given by
\[
I_k(g)=
\frac{l(l-1)(l+1)}
{k(k-1)(2k-l+1)}.
\]
In particular, \(I_k(g)\) is strictly decreasing in \(k\).
\end{proposition}

\begin{proof}
By Proposition~\ref{P:g}, under the standard normal null,
\[
g^{\mathrm{inc}}(x_1,\ldots,x_k)
=
x_1x_k\, e_{l-2}(x_2,\ldots,x_{k-1}),
\]
where \(e_{l-2}\) is short for the elementary symmetric polynomial in the middle variables and of degree \(l-2\).
Since distinct square-free monomials in independent standard normal variables are
orthogonal in \(L^2\), we have $\displaystyle\|g^{\mathrm{inc}}\|^2=\binom{k-2}{l-2}$.

Dividing by the expression in Lemma~\ref{L:g_variance_general} and simplifying, 
we obtain the claimed expression, 
which is clearly strictly decreasing in \(k\).
\end{proof}


A similar result also follows for the trend statistic $K$.

\begin{proposition}
Consider the trend statistic \(K\) with power weights $\alpha_\Delta=\Delta^p$, for $ p\ge 0$.
Then the relative incremental index is decreasing in $k$ and is given by
\[
I_k(K)
=
\frac{(k-1)^{2p}}
{\sum_{\Delta=1}^{k-1}(k-\Delta)^2\Delta^{2p}},
\]
 Moreover,
\[
I_k(K) \asymp k^{-3}.
\]
\end{proposition}

\begin{proof}
By the expression for \(K_1\) obtained above and Lemma~\ref{L:sgn}, we compute $\|K_1\|^2=\displaystyle\frac{\alpha_{k-1}^2}{3}$.
We obtain the formula
\[
I_k(K)
=
\frac{\alpha_{k-1}^2}
{\sum_{\Delta=1}^{k-1}(k-\Delta)^2\alpha_\Delta^2}
\]
Put \(a=2p\) and \(b=k-1\), substitute \(\alpha_\Delta=\Delta^p\) and reparametrize to get
\[
I_k(K)^{-1}
=
\sum_{r=1}^{b} r^2
\left(\frac{b+1-r}{b}\right)^a.
\]
For each fixed \(r\), the factor
\[
\left(\frac{b+1-r}{b}\right)^a
\]
is nondecreasing in \(b\), and passing from \(b\) to \(b+1\) also adds a new nonnegative term. Hence \(I_k(K)^{-1}\) is increasing in \(k\), and therefore \(I_k(K)\) is decreasing.
Finally, a Riemann-sum argument gives
\[
\sum_{\Delta=1}^{k-1}(k-\Delta)^2\Delta^{2p}
\asymp k^{2p+3},
\]
while the numerator is of order \(k^{2p}\). Thus
\[
I_k(K)\asymp k^{-3}.
\]
\end{proof}


The two preceding calculations reveal a common phenomenon. Although the correlation statistic \(g\) and the trend statistic \(K\) are structurally different, their relative incremental contributions both decrease at the same asymptotic rate, of order \(k^{-3}\), in natural nested regimes: fixed degree \(l\) in the polynomial case, and bounded or moderately varying weights \(\alpha_\Delta\) in the trend case. This reflects the same mechanism in both examples: the full statistic aggregates many lower-span contributions, while the incremental component isolates only the full-span endpoint contribution. Thus, increasing the window size eventually adds less genuinely new information relative to the accumulated lower-span structure.



\subsection{Local asymptotics}\label{sS:local_asymp}
To complement the design principles above, we briefly examine local asymptotic behavior
for the proposed window statistics. The calculations follow standard LAN/Pitman-type arguments under contiguous alternatives (see, for example, Chapter 7 of~\cite{vdVaart1998}). These calculations identify the relevant detection
scales under contiguous alternatives and help motivate the local-power simulations
reported in the next section. We first consider the polynomial correlation statistics,
and then turn to the localized trend statistic.

\subsection*{Local asymptotics for $g$ under Gaussian AR(1): first- vs second-order sensitivity}

We begin with local alternatives of the form
\[
\theta_n=\theta_0+c\,n^{-\gamma}, \qquad c\neq 0,\ \gamma>0,
\]
and ask which values of $\gamma$ yield a nondegenerate asymptotic regime. The
classical Pitman case corresponds to $\gamma=1/2$; see, for example, [11, 21]. More
generally, the relevant rate depends on the order of the first nonvanishing term in the
mean expansion under the alternative. To illustrate this, consider a stationary Gaussian
AR(1) model
\begin{equation}\label{E:AR_model}
X_t=\rho X_{t-1}+\varepsilon_t,\qquad \varepsilon_t\sim N(0,1-\rho^2),
\end{equation}
where the innovations \(\varepsilon_t\) are independent, independent of the past, and
\(\varepsilon_t\sim N(0,1-\rho^2)\); the process is initialized in its stationary
distribution, so that \(X_t\sim N(0,1)\).
Thus, $E[X_t]=0$ and $\operatorname{Cov}(X_s,X_t)=\rho^{|s-t|}$. Since the process is stationary and geometrically mixing for $|\rho|<1$, standard central limit theorems for functionals of such sequences apply. For a fixed window
function $f\in \mathcal{H}_k$, write
\[
m_f(\rho):=E_\rho[f(X_1,\dots,X_k)].
\]
The local behavior of the normalized overlapping sum is then determined by the
expansion of $m_f(\rho)$ at $\rho=0$.



\begin{proposition}[Second-order local behavior of the polynomial correlation statistics]
Let $g$ be the symmetric window function given in (\ref{E:g l k}) 
and let $g_1$ be its incremental component as in Proposition~4.1.
Under the Gaussian AR(1) model, $m_g'(0)=m_{g_1}'(0)=0$, and
\[
m_g(\rho)=\binom{k-2}{2}\rho^2+O(\rho^3),
\qquad
m_{g_1}(\rho)=\rho^2+O(\rho^3).
\]
Consequently, under the second-order local scaling $\rho_n=c\,n^{-1/4}$,
\[
\frac{(n-k+1)^{-1/2}\sum_{i=1}^{n-k+1} g(Y_i)}{\sigma(g)}
\Rightarrow
N\!\left(\frac{\binom{k-2}{2}c^2}{\sigma(g)},\,1\right),
\]
and 
\[
\frac{(n-k+1)^{-1/2}\sum_{i=1}^{n-k+1} g_1(Y_i)}{\sigma(g_1)}
\Rightarrow
N\!\left(\frac{c^2}{\sigma(g_1)},\,1\right).
\]
\end{proposition}

\noindent Under the i.i.d.\ $N(0,1)$ null, the normalizing constants are
\[
\sigma^2(g_1)=\binom{k-2}{2},
\qquad
\sigma^2(g)=\binom{k}{4}+2\binom{k}{5}.
\]
These two identities follow from
Proposition~\ref{P:g_Asymp_var} and Lemma~\ref{L:g_variance_general} respectively.

\noindent
\emph{Benchmark (first-order serial statistic).}
For comparison, consider the adjacent-pair window function
$f_{\mathrm{adj}}(x_1,\ldots,x_k)=\sum_{t=2}^k x_{t-1}x_t$.
Then $m_{f_{\mathrm{adj}}}(\rho)=(k-1)\rho$, so under the usual contiguous scaling $\rho_n=c/\sqrt{n}$,
\[
\frac{(n-k+1)^{-1/2}\sum_{i=1}^{n-k+1} f_{\mathrm{adj}}(Y_i)}{\sigma(f_{\mathrm{adj}})}
\Rightarrow
N\!\left(\frac{(k-1)c}{\sigma(f_{\mathrm{adj}})},\,1\right).
\]

\medskip
These calculations highlight a practical design principle: the local detection rate depends on the
lowest-order nonvanishing term in the expansion of $m_f(\rho)$ at $\rho=0$.
Even-order polynomial detectors such as $g$ (and hence $g_1$) are second-order at the i.i.d.\ null,
whereas window functions containing adjacent products attain first-order sensitivity.

\subsection*{Local asymptotics for the trend statistic $K$}

We consider local drift alternatives of the form
\begin{equation}\label{E:drift_alt}
X_t = \varepsilon_t + \beta_n\Bigl(t - \frac{n+1}{2}\Bigr), \qquad 1 \le t \le n,
\end{equation}
where $\{\varepsilon_t\}$ are i.i.d.\ with a continuous distribution function $F$ and density $f$, and where $\beta_n \to 0$. 
The centering ensures that the deterministic drift has zero average over the sample, while preserving the pairwise differences
\[
X_{t+\Delta} - X_t = (\varepsilon_{t+\Delta} - \varepsilon_t) + \beta_n \Delta.
\]

\begin{proposition}
Let
\[
m_K(\beta) = \mathbb{E}_\beta\bigl[K(X_1,\ldots,X_k)\bigr],
\]
where the expectation is taken under the local window model $X_j = \varepsilon_j + \beta j$, $1 \le j \le k$. 
Assume that the density $h$ of $\varepsilon_2 - \varepsilon_1$ is continuous at $0$. Then, as $\beta \to 0$,
\[
m_K(\beta)
=
2 h(0)\, \beta \sum_{\Delta=1}^{k-1} (k-\Delta)\alpha_\Delta \Delta
+ o(\beta).
\]
In particular, $K$ is first-order at the null, and the appropriate local scaling is
\[
\beta_n = c\, n^{-1/2}.
\]
Under this scaling, the normalized sliding-window statistic based on $K$ admits a nondegenerate local asymptotic power function.
\end{proposition}

\begin{proof}
For each fixed $\Delta$ and $j$, write
\[
X_{j+\Delta} - X_j = D_{j,\Delta} + \beta \Delta,
\qquad
D_{j,\Delta} = \varepsilon_{j+\Delta} - \varepsilon_j.
\]
Then
\[
\mathbb{E}_\beta\bigl[\operatorname{sgn}(X_{j+\Delta} - X_j)\bigr]
=
\mathbb{E}\bigl[\operatorname{sgn}(D_{j,\Delta} + \beta \Delta)\bigr].
\]
Since $D_{j,\Delta}$ has a symmetric distribution with density $h$ at zero,
\[
\mathbb{E}\bigl[\operatorname{sgn}(D_{j,\Delta} + a)\bigr]
=
2 h(0)\, a + o(a),
\qquad a \to 0.
\]
Substituting $a = \beta \Delta$ and summing over $j$ and $\Delta$ yields the claimed formula.
\end{proof}

\paragraph{Remark on the projected statistic.}
The corresponding projected statistic $K_1$ does not appear to admit a standard Pitman-type local asymptotic regime under drift alternatives. 
Extensive simulations indicate that, under scalings of the form $\beta_n = c\,n^{-\gamma}$, the rejection probabilities do not stabilize as $n$ increases but instead continue to grow, suggesting the absence of a nondegenerate local limit. 
This behavior contrasts sharply with that of $K$, and reflects the highly localized nature of $K_1$, which relies on a single endpoint comparison rather than aggregation across lags.

The preceding local calculations serve mainly to identify the relevant detection
scales for the window statistics considered here. In particular, the quartic
correlation statistics $g$ and $g_1$ are second-order at the i.i.d.\ null under
Gaussian AR(1) perturbations, leading to the local scale $\rho_n \asymp n^{-1/4}$,
whereas the trend statistic $K$ is first-order under drift alternatives, with the
usual scale $\beta_n \asymp n^{-1/2}$. The expansion details for the correlation
statistics are outlined in Appendices A and B. These conclusions guide the simulation
settings used in the next section.


\section{Simulation study}

We investigate the finite-sample performance of the proposed statistics under
correlation and trend alternatives. The simulation settings are chosen to reflect the
local asymptotic considerations of Section~\ref{sS:local_asymp}: for the correlation statistics $g$
and $g_1$, we consider Gaussian AR(1) alternatives, while for the trend statistic we
consider drift alternatives. All reported rejection probabilities are based on $3000$
Monte Carlo replications. We also verified by simulation that all statistics are
correctly calibrated under the null: across all configurations considered, the
empirical rejection probabilities remained within Monte Carlo error of the nominal
$5\%$ level.

\subsection{Correlation statistics}

We first consider the statistics $g$ and $g_1$ under AR(1) alternatives in (\ref{E:AR_model}).

Under fixed alternatives, the statistic $g$ exhibits strong and monotone power as $\rho$ increases, with substantial gains as the sample size grows. The projected statistic $g_1$ is uniformly less powerful, reflecting its role as an incremental component isolating higher-order contributions.

Under local alternatives $\rho_n = c\,n^{-1/4}$, both statistics display stable rejection probabilities across $n$ for fixed $c$, while increasing with $c$. In contrast, under $\rho_n = c\,n^{-1/2}$, rejection probabilities decrease toward the nominal level. This confirms the second-order nature of the statistics and validates the scaling $n^{-1/4}$.


\begin{table}[ht]
\centering
\begin{minipage}{0.48\textwidth}
\centering
\begin{tabular}{c|ccccc}
\hline
$c \backslash n$ & 100 & 200 & 400 & 600 & 800 \\
\hline
0.5 & 0.095 & 0.089 & 0.086 & 0.087 & 0.083 \\
1.0 & 0.237 & 0.235 & 0.203 & 0.190 & 0.192 \\
1.5 & 0.526 & 0.520 & 0.536 & 0.517 & 0.525 \\
\hline
\end{tabular}
\caption{Empirical power of $g_1$ under local alternatives $\rho_n = c\,n^{-1/4}$. The rejection probabilities remain approximately stable across $n$, confirming the second-order local asymptotic regime.}
\label{tab:g1_correct_scaling}
\end{minipage}
\hfill
\begin{minipage}{0.48\textwidth}
\centering
\begin{tabular}{c|ccccc}
\hline
$c \backslash n$ & 100 & 200 & 400 & 600 & 800 \\
\hline
0.5 & 0.054 & 0.059 & 0.053 & 0.057 & 0.051 \\
1.0 & 0.069 & 0.066 & 0.063 & 0.054 & 0.051 \\
1.5 & 0.093 & 0.072 & 0.066 & 0.070 & 0.059 \\
\hline
\end{tabular}
\caption{Empirical power of $g_1$ under local alternatives $\rho_n = c\,n^{-1/2}$. The rejection probabilities decrease toward the nominal level, indicating that this scaling is asymptotically too weak.}
\label{tab:g1_wrong_scaling}
\end{minipage}
\end{table}


\begin{figure}[ht]
\centering
\includegraphics[width=0.95\textwidth]{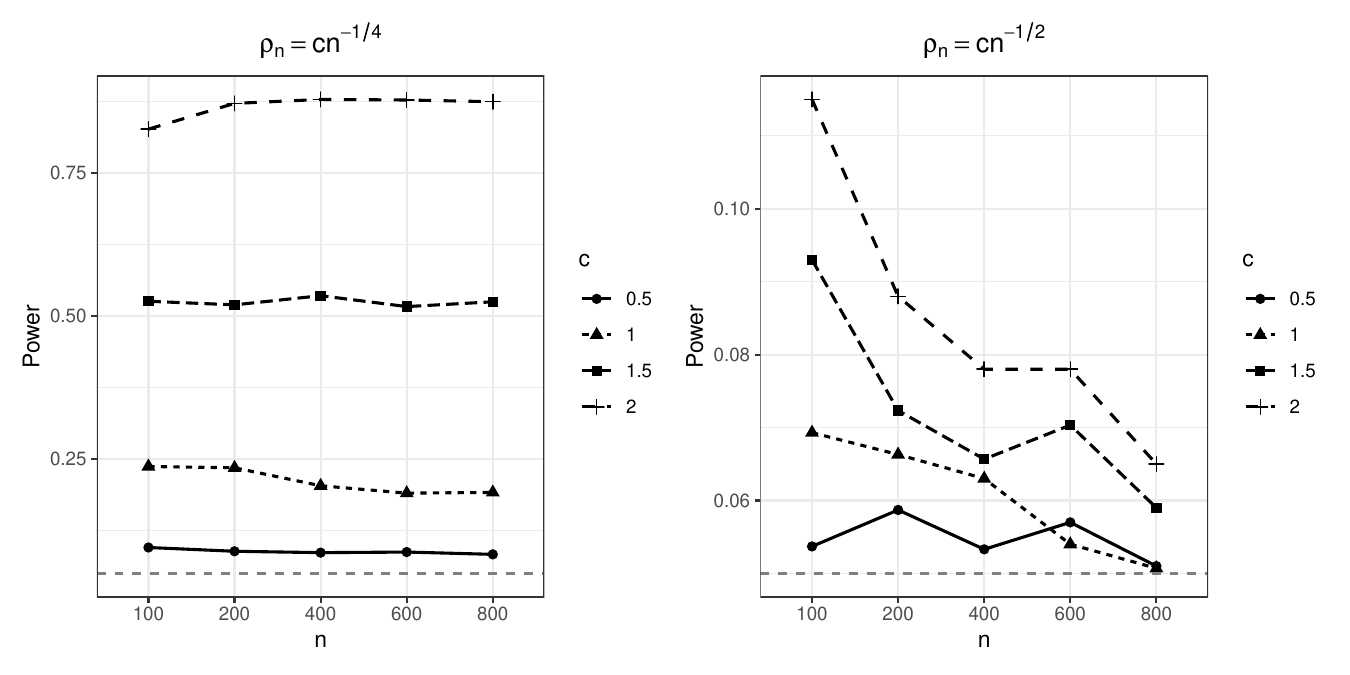}
\caption{Empirical power of $g_1$ under local AR(1) alternatives. Left: correct scaling $\rho_n = c\,n^{-1/4}$, showing approximately stable rejection probabilities across $n$. Right: incorrect scaling $\rho_n = c\,n^{-1/2}$, showing rejection probabilities drifting toward the nominal level.}
\label{fig:g1_local_power}
\end{figure}

\begin{figure}[ht]
\centering
\includegraphics[width=0.95\textwidth]{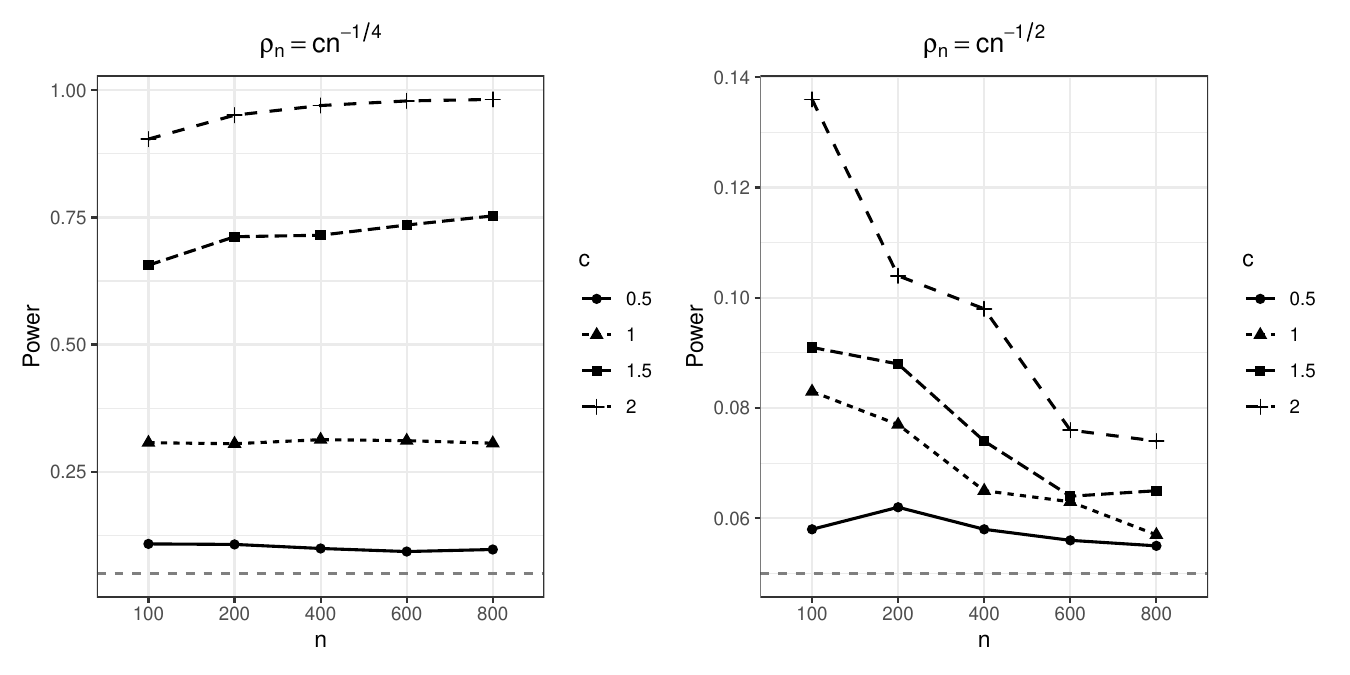}
\caption{Empirical power of the unprojected $g$ under local AR(1) alternatives. Left: correct scaling $\rho_n = c\,n^{-1/4}$, showing approximately stable rejection probabilities across $n$. Right: incorrect scaling $\rho_n = c\,n^{-1/2}$, showing rejection probabilities drifting toward the nominal level.}
\label{fig:g1_local_power}
\end{figure}

\subsection{Trend statistics and the role of weighting}

We examine the behavior of the unprojected  statistic $K(x_1,\ldots,x_k)$ 
under local drift alternatives of the form given in (\ref{E:drift_alt}) with normally distributed errors $\varepsilon_t \sim N(0,1)$.
\smallskip
\paragraph{Local regime.}
Under the scaling $\beta_n = c\,n^{-1/2}$, the rejection probabilities remain stable across $n$ for fixed $c$, confirming the first-order nature of $K$. A representative set of results for the baseline choice $(\alpha_1,\alpha_2,\alpha_3)=(1,1,1)$ is reported in Table~\ref{tab:K_equal_weights}.
\smallskip

\paragraph{Effect of weighting.}
The coefficients $\alpha_\Delta$ allow different lag separations to contribute unequally. Under drift alternatives, larger lags carry stronger deterministic shifts, suggesting that increasing weights with $\Delta$ may improve power. This effect is clearly visible in Figure~\ref{fig:K_weights}, where the increasing scheme $(1,2,3)$ dominates the baseline $(1,1,1)$ across all values of $c$. 
More aggressive weighting schemes further improve the power, but with diminishing returns: $(1,2,3)$ and $(1,2,4)$ perform similarly, indicating that most of the available signal is already captured by moderate weighting, and that additional emphasis on the largest lag yields only marginal gains.

\smallskip
\paragraph{Lack of aggregation.}
In contrast, when only the largest lag is retained, as in $(\alpha_1,\alpha_2,\alpha_3)=(0,0,1)$, the statistic loses the benefit of aggregation and exhibits substantially weaker power, as shown in Table~\ref{tab:K_endpoint}.
\smallskip
\paragraph{Interpretation.}
These results demonstrate that the coefficients $\alpha_\Delta$ play a fundamental role: they control how signal accumulates across lags. While $K$ is first-order regardless of the choice of weights, its efficiency depends critically on the degree of aggregation, with increasing weight schemes providing the strongest performance under drift.

\begin{table}[ht]
\centering
\begin{minipage}{0.48\textwidth}
\centering
\begin{tabular}{c|cccc}
\hline
$c \backslash n$ & 100 & 200 & 400 & 800 \\
\hline
0.25 & 0.451 & 0.470 & 0.448 & 0.464 \\
0.50 & 0.663 & 0.663 & 0.669 & 0.667 \\
0.75 & 0.834 & 0.852 & 0.846 & 0.851 \\
1.00 & 0.947 & 0.954 & 0.948 & 0.956 \\
\hline
\end{tabular}
\caption{Empirical power of $K$ under local drift alternatives $\beta_n = c\,n^{-1/2}$ with weights $(1,1,1)$. The rejection probabilities are stable across $n$ and increase with $c$.}
\label{tab:K_equal_weights}
\end{minipage}
\hfill
\begin{minipage}{0.48\textwidth}
\centering
\begin{tabular}{c|cccc}
\hline
$c \backslash n$ & 100 & 200 & 400 & 800 \\
\hline
0.25 & 0.004 & 0.005 & 0.002 & 0.003 \\
0.50 & 0.020 & 0.024 & 0.022 & 0.021 \\
0.75 & 0.090 & 0.104 & 0.100 & 0.112 \\
1.00 & 0.269 & 0.298 & 0.288 & 0.308 \\
\hline
\end{tabular}
\caption{Empirical power of $K$ under local drift alternatives $\beta_n = c\,n^{-1/2}$ with weights $(0,0,1)$. The absence of aggregation leads to substantially weaker performance.}
\label{tab:K_endpoint}
\end{minipage}
\end{table}

\begin{figure}[ht]
\centering
\includegraphics[width=0.6\textwidth]{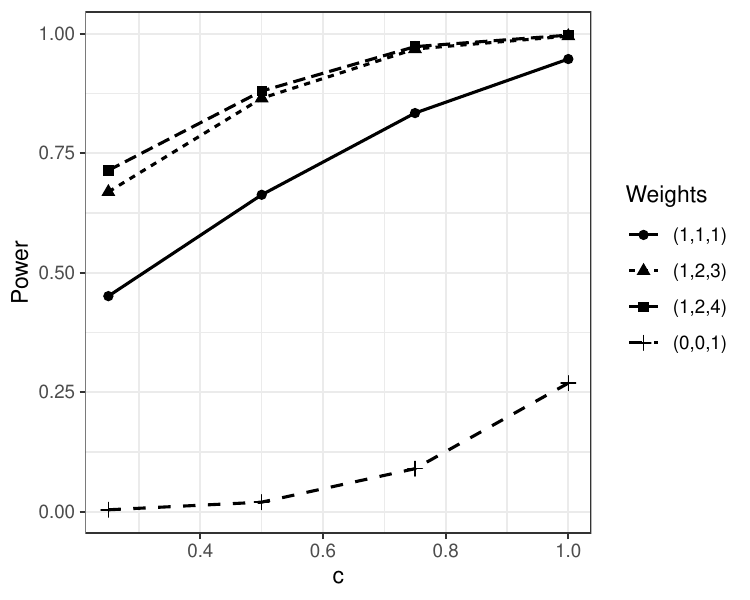}
\caption{Effect of weighting coefficients $\alpha_\Delta$ on the power of $K$ under local drift alternatives ($n=100$). Increasing weights enhance sensitivity, while endpoint-only weighting leads to substantial loss of power.}
\label{fig:K_weights}
\end{figure}

\subsection{Projected trend statistic}

Finally, we consider the projected statistic $K_1$, which isolates the endpoint contribution.

\paragraph{Empirical local behavior.}
In contrast to the full statistic $K$, the projected statistic $K_1$ does not exhibit a stable local asymptotic regime under drift alternatives. To investigate this, we examined a range of candidate scalings of the form
\[
\beta_n = c\,n^{-\gamma},
\]
including $\gamma=1/2$, $\gamma=1/4$, and $\gamma=1/6$.

The resulting rejection probabilities fail to display the characteristic stabilization across $n$ associated with a standard Pitman regime. For $\gamma=1/2$, the power remains close to the nominal level except for relatively large values of $c$, indicating that the signal is too weak. For $\gamma=1/6$, the power increases rapidly toward one, indicating that the signal is too strong. Intermediate scalings, such as $\gamma=1/4$, still produce rejection probabilities that increase with $n$ rather than stabilizing.

Figure~\ref{fig:K1_scaling} summarizes the contrasting behavior across the three candidate scalings.

\begin{figure}[ht]
\centering
\includegraphics[width=\textwidth]{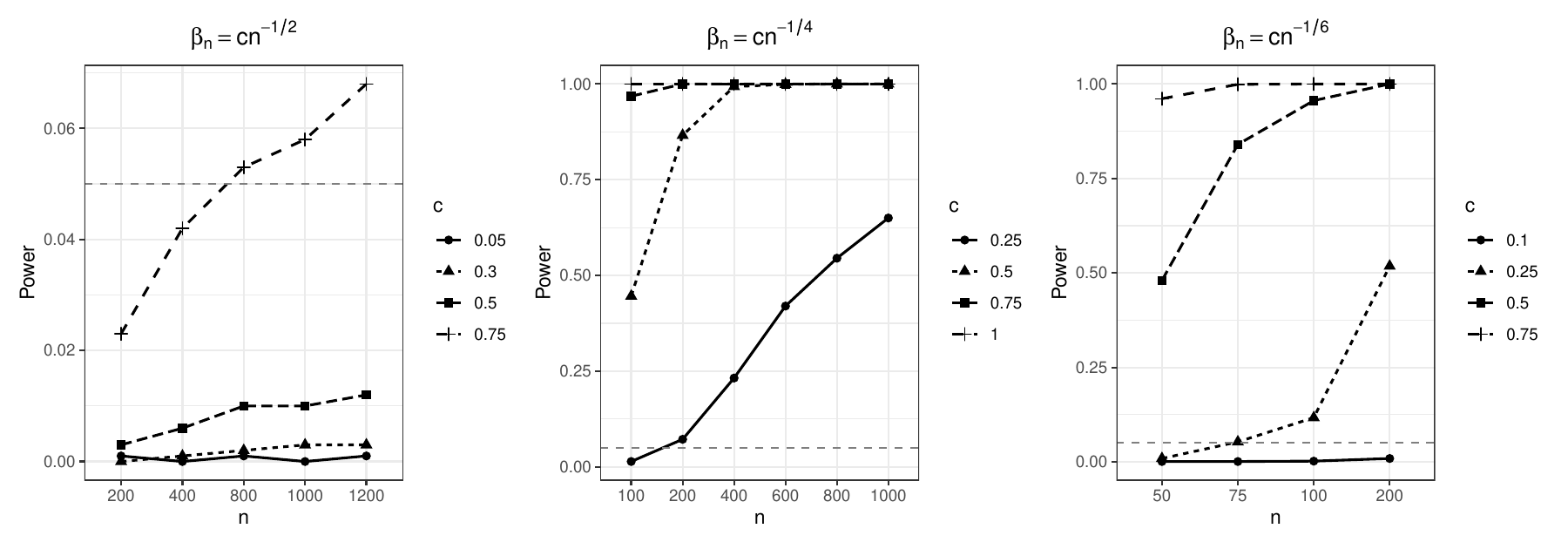}
\caption{Empirical power of $K_1$ under local drift alternatives $\beta_n=c\,n^{-\gamma}$ for three candidate scalings. Left: $\gamma=1/2$, where the signal remains weak. Middle: $\gamma=1/4$, where the power continues to increase with $n$ rather than stabilizing. Right: $\gamma=1/6$, where the power increases too rapidly toward one. No choice of $\gamma$ yields the stabilization characteristic of a standard Pitman local asymptotic regime.}
\label{fig:K1_scaling}
\end{figure}

\paragraph{Interpretation.}
Taken together, these findings provide strong empirical evidence that $K_1$ does not appear to admit a nondegenerate Pitman local asymptotic regime under drift alternatives. This behavior contrasts sharply with that of $K$, and reflects the highly localized structure of $K_1$, which relies on a single endpoint comparison rather than aggregation across lags.

\subsection{Summary}

The simulations reveal three distinct regimes:

\begin{itemize}
\item $g$ and $g_1$: second-order behavior with scaling $n^{-1/4}$;
\item $K$: first-order behavior with scaling $n^{-1/2}$;
\item $K_1$: absence of a standard local asymptotic regime.
\end{itemize}




\section{Concluding Remarks}

We have developed a framework for analyzing overlapping window statistics through an
orthogonal decomposition of $\mathcal H_k$, together with an auxiliary decomposition that
facilitates explicit calculations. This viewpoint isolates the incremental contribution
introduced by enlarging the window and yields a spectral representation of the asymptotic
variance. The examples considered here illustrate how the method applies both to localized
polynomial correlation statistics and to localized rank-based trend statistics.

Two directions seem especially natural for further study. The first concerns scale selection. The relative incremental index introduced in this work provides a quantitative criterion for choosing the window size, and the examples suggest that this index may decrease with \(k\), in some cases at a rate of order \(k^{-3}\). A natural question is to identify general classes of window functions, together with conditions on their structure and weighting, under which such monotonicity or decay rates hold. Establishing a unifying theorem in this direction would clarify the mechanisms governing how dependence information accumulates across scales.
The second direction concerns extensions beyond the i.i.d.\
setting. In the independent case, the incremental component is especially striking in
that it lies in the eigenspace corresponding to eigenvalue $1$ and is not contaminated
by a kernel component. It is therefore of clear theoretical and practical interest to
determine to what extent such a clean structure survives under dependence.


\appendix
\section{Mean expansion under Gaussian AR(1)}


We sketch the derivation of the expansion
\[
m_{g_1}(\rho)=\mathbb{E}_\rho[g_1(Y_1)]=\rho^2+O(\rho^3)
\]
for the statistic $g_1$ defined in \eqref{E:l=4,k}, under a Gaussian AR(1) model with parameter $\rho$.

Let $X=(X_1,\dots,X_k)$ be a centered Gaussian vector with covariance
\[
\Sigma_{ij}=\rho^{|i-j|}.
\]
Since $(X_i,X_j,X_r,X_s)$ is jointly centered Gaussian, Wick's formula
for Gaussian moments (see, e.g., Janson~\cite{Janson1997}) yields, for distinct indices $a<b<c<d$,
\[
\mathbb{E}[X_aX_bX_cX_d]
=
\Sigma_{ab}\Sigma_{cd}
+
\Sigma_{ac}\Sigma_{bd}
+
\Sigma_{ad}\Sigma_{bc}.
\]

We consider the expansion of $\mathbb{E}_\rho[g_1]$ for small $\rho$. Each covariance term contributes a factor of $\rho^{|i-j|}$, so the leading order is determined by the smallest total exponent across pairings.

In the expression for $g_1$, all terms are of the form
\[
(X_1-\mu)(X_k-\mu)X_iX_j, \qquad 2\le i<j\le k-1.
\]
After centering, the expectation reduces to evaluating $\mathbb{E}[X_1X_kX_iX_j]$.

Applying Wick's formula, we obtain
\[
\mathbb{E}[X_1X_kX_iX_j]
=
\Sigma_{1k}\Sigma_{ij}
+
\Sigma_{1i}\Sigma_{kj}
+
\Sigma_{1j}\Sigma_{ki}.
\]

The first term is of order $\rho^{k-1+|i-j|}$ and is negligible for small $\rho$.
The remaining two terms are of order $\rho^{(i-1)+(k-j)}$ and $\rho^{(j-1)+(k-i)}$, respectively.

The smallest possible exponent is $2$, which occurs only when $\{i,j\}=\{2,k-1\}$.
In that case,
\[
(i-1)+(k-j)=1+1=2.
\]

Thus, exactly one unordered pair contributes at order $\rho^2$, and we obtain
\[
m_{g_1}(\rho)=\rho^2+O(\rho^3).
\]

This shows that $g_1$ is second-order at the independence null.
\qed

\section{Local scaling for second-order alternatives}

We justify the scaling $\rho_n = c\,n^{-1/4}$ for the statistic $g_1$. Let
$T_n(f)=\displaystyle\frac{S_n(f)}{\sigma(f)}$, 
and define
\[
m(\rho)=\mathbb{E}_\rho[f(Y_1)].
\]

Assume that $m(\rho)$ admits the expansion $m(\rho)=a\rho^2+o(\rho^2)$ 
as $\rho\to 0$, with $a\ne 0$.

\begin{proposition}
Let $\rho_n = c\,n^{-1/4}$. Then
\[
\mathbb{E}_{\rho_n}[S_n(f)]
=
\sqrt{n}\,m(\rho_n)
=
a c^2 + o(1).
\]
Consequently,
\[
T_n(f)\;\Rightarrow\; N\!\left(\frac{a c^2}{\sigma(f)},\,1\right).
\]
\end{proposition}

\begin{proof}
By linearity of expectation,
\[
\mathbb{E}_{\rho_n}[S_n(f)]
=
\frac{1}{\sqrt{n}}\sum_{i=1}^{n-k+1}\mathbb{E}_{\rho_n}[f(Y_i)]
=
\sqrt{n}\,m(\rho_n)+o(1).
\]

Using the expansion of $m(\rho)$ and multiplying by 
$\sqrt{n}$ gives
\[
\mathbb{E}_{\rho_n}[S_n(f)]
=
a c^2 + o(1).
\]

Under the null, $S_n(f)$ satisfies a central limit theorem with variance $\sigma^2(f)$.
Since $\rho_n\to 0$, the variance remains asymptotically unchanged, yielding the stated convergence.
\end{proof}

\medskip

If instead $\rho_n=c\,n^{-1/2}$, then
\[
\sqrt{n}\,m(\rho_n)=O(n^{-1/2})\to 0,
\]
and the limiting distribution is the same as under the null. This explains the collapse of power under the usual contiguous scaling.


\begin{thebibliography}{99}
\bibitem{AA04} A. Alhakim, \textit{On the eigenvalues and
eigenvectors of an overlapping Markov chain}, Probability
Theory and Related Fields, 128 (2004), pp. 589-605.

\bibitem{AA2024} A. Alhakim, \textit{Hadamard matrices, quaternions, and the Pearson
chi-square statistic}, Statistical Papers, 65(8), (2024), pp. 5273–5291.

\bibitem{AKM} A. Alhakim, J. Kawczak, and S. Molchanov, \textit{On
the class of nilpotent Markov chains, I: the spectrum of
covariance operator}, Markov Processes and Related Fields
10 (2004), pp. 629-652.


\bibitem{Cabilio2013} P. Cabilio, Y. Zhang, X. Chen, \textit{ Bootstrap rank tests for trend in time series}, Environmetrics, V 24, 8 (2013), pp. 537 - 549.




\bibitem{Good1953} Good, I.J., (1953) ``\textit{The serial test for sampling numbers and other tests of randomness}''. In Proceedings of Cambridge Philosophical Society, 49, 276-284.

\bibitem{GordinLifsic1978} M. I. Gordin and  B. A. Lif\v{s}ic, \textit{Central limit theorem for stationary Markov processes}, Soviet Mathematics Doklady, 19, (1978), pp. 392--394.

\bibitem{HoeffdingRobbins1948}
W. Hoeffding and H. Robbins,
The central limit theorem for dependent random variables,
\emph{Duke Mathematical Journal} 15 (1948), 773--780.

\bibitem{Janson1997} S. Janson, Gaussian Hilbert Spaces, Cambridge University Press, 1997.

\bibitem{Kendall1955} M. G. Kendall,  Rank Correlation Methods,  Griffin, 1955.


\bibitem{LeCam1986} L. Le Cam, Asymptotic Methods in Statistical Decision Theory, Springer, 1986.

\bibitem{LeCamYang2000} L. Le Cam and G. L. Yang, Asymptotics in Statistics: Some Basic Concepts, Springer, 2000.

\bibitem{LehmannRomano2005} E.L. Lehmann and J.P. Romano, Testing Statistical Hypotheses, 3rd. Ed., Springer, 2005.

\bibitem{LSW2002} P. L'Ecuyer, R. Simard and S. Wegenkittl, \textit{Sparse Serial Tests of Uniformity for Random Number Generators}, SIAM Journal On Scientific Computing, 24, 2 (2002), pp. 652-668.

\bibitem{Mann1945} H.B. Mann, \textit{Nonparametric tests against trend}, Econometrica, Journal of the Econometric Society, (1945) pp. 245-259.

\bibitem{MeynTweedie2009} S. P. Meyn,  and  R. L. Tweedie, Markov Chains and Stochastic Stability, 2nd. Ed., Cambridge University Press, 2009.
 

\bibitem{gM85} G. Marsaglia, \textit{A Current View of Random Number Generators}, Computer Science and Statistics, Elsevier Science Publisher B.V., North-Holland, 1985.


\bibitem{Pincus1991} S. M. Pincus, \textit{Approximate entropy as a measure of system complexity},  Proceedings of the National Academy of Science, USA,  88, (1991), pp. 2297–2301.

\bibitem{PS96} S.M.  Pincus and B.H. Singer, \textit{Randomness and Degrees of
Irregularity}. Proceedings of the National Academy of Science, USA, 93 (1996), pp. 2083--2088.

\bibitem{Serfling1980} R. J. Serfling, Approximation Theorems of Mathematical Statistics, Wiley, New York, 1980.

\bibitem{XT07} X. Xu and W. Tsang, \textit{An Empirical Study on the Power of the Overlapping Serial Test},  proc. Asia Simulation Conference 2007,  J.W. Park, T.G. Kim and Y. B. Kim, Eds. Seoul, Korea, (2007), pp. 298-306.

\bibitem{vdVaart1998} Aad W. van der Vaart, Asymptotic Statistics, Cambridge University Press, 1998.





\end{thebibliography}
\end{document}